\documentclass[11pt]{amsart}   
\usepackage{amssymb,xy}
\usepackage{amsthm, amsmath}
\usepackage[latin1]{inputenc}
\usepackage[T1]{fontenc}
\xyoption{all}
\theoremstyle{plain}

\newtheorem{theorem}{Theorem}[section]
\newtheorem{lemma}[theorem]{Lemma}
\newtheorem{proposition}[theorem]{Proposition}
\newtheorem{corollary}[theorem]{Corollary}

\theoremstyle{definition}

\theoremstyle{remark}

\newtheorem{remark}[theorem]{Remark}

\title{Determinants of $(-1,1)$-matrices of the skew-symmetric type: a cocyclic approach}
\author{
V. \'Alvarez
\and
J. A. Armario
\and
M. D. Frau
\and
F. Gudiel
}
\date{}

\address{\small \rm  Depto Matemática Aplicada I\\
 Universidad de Sevilla\\ Avda. Reina Mercedes s/n 41012 Sevilla\\ Spain.}

\email{valvarez@us.es}
\email{armario@us.es}

\email{mdfrau@us.es}
\email{gudiel@us.es}

\begin{document}

\begin{abstract}
An $n$ by $n$ skew-symmetric type $(-1,1)$-matrix $K=[k_{i,j}]$ has $1$'s on the main diagonal and $\pm 1$'s elsewhere with $k_{i,j}=-k_{j,i}$.
The largest possible determinant of such a matrix $K$ is an interesting problem.
The literature is extensive for $n\equiv 0 \mod  4$ (skew-Hadamard matrices), but  for $n\equiv 2\mod 4$ there are few results  known for this question.
In this paper we approach this problem constructing cocyclic matrices over the dihedral group of $2t$ elements, for $t$ odd,  which are equivalent to  $(-1,1)$-matrices of skew type. Some explicit calculations have been done up to $t=11$.
To our knowledge, the upper bounds on the maximal determinant in orders 18 and 22 have been improved.
\end{abstract}
\keywords{$(-1,1)$-matrix of skew type\and Cocyclic matrices \and  Maximal determinants}

\maketitle
\section{Motivation of the problem - introduction }

Let $g(n)$  denote the maximum  determinant of all   $n\times n$  matrices with elements $\pm 1$. Here and throughout this paper, for convenience, when we say determinant we mean the absolute value of the determinant. The question of finding $g(n)$ for any integer $n$ is an old one which remains unasnwered in general. We ignore here the trivial cases $n=1,2$. In 1893 Hadamard gave the bound $n^{n/2}$ for $g(n)$. This bound can be attained only if $n$ is a multiple of $4$. A matrix that attains it is called a {\it Hadamard matrix}, and it is an outstanding conjecture that one exists for any multiple of $4$. At the time of writing, the smallest order for which the existence of a Hadamard matrix is in question is 668. If $n$ is not a multiple of $4$, $g(n)$ is not  known in general, but tighter bounds exist. For $n\equiv 2 \mod 4$, Ehlich \cite{Ehl64} and independently Wojtas \cite{Woj64} proved that
 \begin{equation} \label{ewb1}
 g(n)\leq (2n-2)(n-2)^{\frac{1}{2}n-1}.
 \end{equation}
Moreover, in order for equality to hold, it is required that there exists a $(-1,1)$-matrix $M$ of order $n$ such that $MM^T=\left( \begin{array}{cc} L&0\\ 0&L \end{array}\right)$, where $L=(n-2)I_{\frac{n}{2}}+2J_{\frac{n}{2}}$. Here, as usual, $I_n$ denotes the identity matrix of order $n$, and $J_n$ denotes the $n\times n$ matrix all of whose entries are equal to one. In these circumstances, it may be proved that, in addition, $2n-2$ is the sum of two squares,  a condition which is believed to be sufficient  (order 138 is the lowest for which the question has not been settled yet, \cite{FKS04}). The interested reader is addressed to \cite{KO06} and the website \cite{OS05} for further information on what is known about maximal determinants.

There is a companion theory for matrices with $0$'s on the main diagonal and $\pm 1$ elsewhere.
Let $f(n)$ denote the maximum determinant of all $n\times n$ matrices   with elements $0$ on the main diagonal and $\pm 1$ elsewhere.
It is well-known that $f(n)\leq (n-1)^{n/2}$. This can be attained only when $n$ is even, and a matrix $C$ which does so is called a {\it conference matrix}. For $n$ odd the question has been hardly worked out  \cite{BKMS97}.


A matrix $M$ is {\it symmetric} if $M=M^T$. A matrix is {\it skew-symmetric} (or {\it skew}) if $M=-M^T$. The following theorem analyzes the structure of the conference matrices \cite[p. 307]{IK06}:

\begin{theorem}
If $C$ is an $n\times n$ conference matrix, then either $n\equiv 0 \mod  4$ and $C$ is equivalent to a skew matrix, or $n\equiv 2 \mod  4$ and $C$ is equivalent to a symmetric matrix and such a matrix cannot exist unless $n-1$ is the sum of two squares: thus they cannot exist for orders $22,34,58,70,78,94.$ The first values for which the existence of symmetric conference matrices is unknown are $n=66,86$.
\end{theorem}

It is known that if $C$ is a skew conference matrix then $H=C+I$ is a Hadamard matrix.  Which is called a {\it skew Hadamard matrix}. Let us point out that $H+H^T=2I$.

Skew conference and skew Hadamard matrices are of great interest because of their elegant structure, their beautiful properties and their applications to Coding Theory, Combinatorial Designs and Cryptography.

Throughout the paper $f_k(n)$  will denote the maximum determinant of all $n\times n$ skew matrices with elements $0$ on the main diagonal and $\pm 1$ elsewhere. Respectively, $g_k(n)$ will denote the maximum determinant of all $(-1,1)$ matrices of skew type of order $n$. Let us observe that $N$ is an $n\times n$ skew matrix with elements $0$ on the main diagonal and $\pm1$ elsewhere if, and only if, $N+I$ is a $(-1,1)$-matrix of skew type with order $n$.

In what follows, we will deal with these two problems, originally posted by P. Cameron in his website \cite{Cam11} and suggested by Dennis Lin:
\begin{enumerate}
 \item Compute $f_k(n)$ and $g_k(n)$ for different integers $n$.
 \item  Let $N$ be an $n\times n$ skew matrix with elements $0$ on the main diagonal and $\pm 1$ elsewhere; decide whether or not the following equivalence holds. It will be called  {\em ``Lin's correspondence"}.
         $$\det N=f_{k}(n)\Longleftrightarrow \det(N+I)=g_k(n), \,\,n\,\mbox{even}.$$
 \end{enumerate}

For $n\equiv 0 \mod 4$, concerning the first question it is conjectured that $f_k(n)=n^{n/2}$ and $g_k(n)=(n-1)^{n/2}$. The first open order is $n=276$.
On the second question, it is known that  $C$ is a  skew conference matrix if and only if $C+I$ is a  skew Hadamard matrix. So the equivalence above holds.

We shall here be concerned with the case $n\equiv 2\mod 4$, $n\neq 2$ and this will be implicitly assumed in what follows.

Concerning $f_k(n)$ there exist the following upper \cite{AF13} and lower \cite{Cam11} bounds,
\begin{itemize}
\item \begin{equation}\label{e-wbound01}
f_k(n)\leq (2n-3)(n-3)^{\frac{1}{2}n-1}
\end{equation}
 and equality  holds if and only if
there exists a skew matrix $N$ with
\begin{equation}\label{idenecsdc}
N N^T=N^TN=\left[\begin{array}{cc} L' & 0 \\ 0 & L' \end{array}\right],
\end{equation}
where $L'=(n-3)I+2J$.
Will Orrick noticed in \cite{Cam11} that equality in (\ref{e-wbound01}) can only hold if  $2n-3=x^2$ where $x$ is an integer.

\item Assuming that a conference matrix of order $n+2 $ exists,  $$f_k(n)\geq (n+1)^{\frac{1}{2}n-1}.$$
\end{itemize}

For $g_k(n)$, we have:
\begin{itemize}
 \item Elich-Wojtas' upper bound (\ref{ewb1}). That is,
    \begin{equation}\label{stew}
    g_k(n)\leq (2n-2)(n-2)^{\frac{1}{2}n-1}.
     \end{equation}
\item
  An analogous lower bound was given by Cameron \cite{Cam11}   under the hypothesis that a conference matrix of order $n+2$ exists
$$ g_k(n)\geq 2(n+2)^{\frac{1}{2}n-1}.$$
\end{itemize}
 Moreover, the following result from \cite{AF13} gives an affirmative answer to question 2,
\begin{theorem}\label{maxdeterdosmundos}
Let $K$ be a $(-1,1)$-matrix of skew type.
$$\mbox{det}\,K=(2n-2)(n-2)^{\frac{1}{2}n-1}\Leftrightarrow \mbox{det}\,(K-I)=(2n-3)(n-3)^{\frac{1}{2}n-1}.$$
\end{theorem}
  This result implies that equality in (\ref{stew}) can only hold if  $2n-3=x^2$ where $x$ is an integer. Therefore, a $(-1,1)$-matrix of skew type reaching Ehlich-Wojtas' bound cannot exist for orders 10, 18, 22, 30 and so on.

An alternative proof for the above lower bounds can be done using a corollary (posted in \cite[Corollary 1]{KMNS12}) of the following result, originally proved by Szollosi.
\begin{theorem}[Szollosi \cite{Szo10}]
Let $M=\left[\begin{array}{cc}
X & Y \\
Z & W
\end{array}\right]$ be an $n\times n$ orthogonal matrix which is $(l,n-l)$-partitioned, i.e., $X$ is $l\times l$, $W$ is $(n-l)\times (n-l)$, $Y$ is $l\times (n-l)$ and $Z$ is $(n-l)\times l$. Then $\det X=\det W$.
\end{theorem}
Considering $C$ a skew conference matrix, $CC^T=(n-1) I$, and $H=C+I$ a skew Hadamard matrix, $HH^t=nI$, in the place of orthogonal matrix,
\begin{corollary}
Let $C=\left[\begin{array}{cc}
X & Y \\
Z & W
\end{array}\right]$ be a $(n+2) \times (n+2)$ skew conference matrix partitioned as above with $l\leq \frac{n}{2}+1$. Then the lower right $(n-l)\times (n-l)$, $l\geq 1$, minor of $C$ is
$$\det W=(n+1)^{\frac{n}{2}+1-l}\det X.$$
Taking $H=C+I$.Then the lower right $(n-l)\times (n-l)$, $l\geq 1$, minor of $C+I$ is
$$\det (W+I)=(n+2)^{\frac{n}{2}+1-l}\det (X+I).$$
 \end{corollary}
\begin{remark}
Let us point out that if $l=2$, the corollary above provides a skew matrix $C$ whose determinant is $(n+1)^{\frac{n}{2}-1}$ and
a $(-1,1)$-matrix of skew type, $C+I$, whose determinant is  $2(n+2)^{\frac{n}{2}-1}$. Therefore, the lower bounds for $f_k$ and $g_k$ are proved.
\end{remark}

When a $n\times n$ determinant is found that attains the relevant one of the above upper bounds, it is immediate that the maximal determinant for that order is just the bound itself. For instance, W. Orrick found $(-1,1)$-matrices of skew type whose determinants reach Ehilich-Wojtas' bound for $n=6, 14, 26$ and $42$. Therefore, by means of Lin's correspondence (Theorem \ref{maxdeterdosmundos}), $f_k(n)$ so does.
Nevertheless when the upper bound is not attained, finding $g_k(n)$ and $f_k(n)$ can be exceedingly difficult.  For $n\leq 30$, orders 18, 22, 30 are unresolved. The case $n=10$ was solved by Cameron, $g_k(10)=64000$ and $f_k(10)=33489$, using a random search and again Lin's correspondence was confirmed.

Two $(-1,1)$-matrices $M$ and $N$ are said to be {\it Hadamard equivalent} or {\it equivalent} if one can be obtained from the other by a sequence of the operations:
\begin{itemize}
\item interchange any pairs of rows and/or columns;
\item multiply any rows and/or columns through by $-1$.
\end{itemize}
In other words, we say that $M$ and $N$ are equivalent if there exist $(0,1,-1)$-monomial matrices $P$ and $Q$ such that $PMQ^T=N$.

 In the early 90s, a surprising link between homological algebra and Ha\-da\-mard matrices \cite{HD94} led to the study of cocyclic Hadamard matrices.
Hadamard matrices of many types are revealed to be (equivalent to) cocyclic matrices \cite{Hor07}. Among them, Sylvester Hadamard matrices, Williamson Hadamard matrices, Ito Hadamard matrices and Paley Hadamard matrices. Furthermore, the
     cocyclic construction is the most uniform construction technique for Hadamard matrices currently known, and
     cocyclic Hadamard matrices
    may consequently provide a uniform approach to the
famous Hadamard conjecture.

The main advantages of the cocyclic framework concerning Hadamard matrices may be summarized in the following facts:

\begin{itemize}

\item The  test to decide whether a cocyclic matrix is Hadamard runs in $O(t^2)$ time, better than the $O(t^3)$ algorithm for usual (not necessarily cocyclic)  matrices.

\item The search space is reduced to the set of cocyclic matrices over a given group (that is, $2^s$ matrices, provided that a basis for
cocycles over $G$ consists of $s$ generators), instead of the whole set of $2^{16t^2} $ matrices of order $4t$ with entries in $\{-1,1\}$.  \end{itemize}

Let us point out that there is some evidence that searching for cocyclic Hadamard matrices, and in particular for $D_{4t}$-cocyclic Hadamard matrices (i.e., cocyclic Hadamard matrices over the dihedral group of $4t$ elements), makes sense. Furthermore, in \cite{AAFG12}, we showed that  the cocyclic technique can  certainly be extended to handle the maximal determinant problem (for matrices of order $n$ with entries in $\{-1,1\}$) at least when $n\equiv 2 \mod 4$. More concretely, we focused on cocyclic matrices over the dihedral group $D_{2 t}$,  with $t$ odd. We provided some algorithms for constructing $D_{2t}$-cocyclic matrices with large determinants and some explicit calculations up to $t=19$.

In this paper, taking as a starting point the cocyclic matrices yielded by the algorithms from \cite{AAFG12}, we check which of those matrices are equivalent to a $(-1,1)$-matrix of skew type, $M$. In this way, we provide a method for constructing $(-1,1)$-matrices of skew type with large determinants. As far as we know, new records for $n=18$ and $22$ have been provided. Besides, we investigate the spectrum of the determinant function for these matrices $M$ and the relationship with the spectrum of the corresponding skew matrices $M-I$. In the light of these computations, we conjecture that when $\det M$ moves in the range $[ 2(n+2)^{\frac{1}{2}n-1},\,(2n-2)(n-2)^{\frac{1}{2}n-1}]$ then $\det (M-I)$ is a monotonic increasing function in the range $[(n+1)^{\frac{1}{2}n-1},\,(2n-3)(n-3)^{\frac{1}{2}n-1}]$. This weighs heavily in favor of  Lin's correspondence.

\section{Cocyclic matrices and $(-1,1)$-matrices with large determinants}
Assume throughout that $G=\{g_1=1,\,g_2,\ldots,g_{n}\}$ is a multiplicative group, not necessarily abelian. Functions
$\psi\colon G\times G\rightarrow \langle -1\rangle\cong {\bf Z}_2$ which satisfy
\begin{equation}\label{condiciondecociclo}
\psi(g_i,g_j)\psi(g_ig_j,g_k)=\psi(g_j,g_k)\psi(g_i,g_jg_k), \quad\forall g_i,g_j,g_k\in G
\end{equation}
are
called (binary) cocycles  (over $G$) \cite{McL95}. A cocycle is a  coboundary
$\partial\phi$ if it is derived from a set mapping $\phi\colon
G\rightarrow \langle -1\rangle$  by
$\partial\phi(a,b)=\phi(a)\phi(b)\phi(ab)^{-1}.$

A cocycle $\psi$ is naturally
displayed as {\it a cocyclic matrix (or $G$-matrix)} $M_\psi$; 
 that is, the entry in the $(i,j)$th position of the cocyclic matrix is $\psi(g_i,g_j)$, for all $1\leq i,j\leq n$.


A cocycle $\psi$ is {\it normalized} if $\psi(1,g_j)=\psi(g_i,1)=1$ for all $g_i,g_j \in G$. The cocyclic matrix coming from a normalized cocycle is called {\it normalized} as well.  Each unnormalized cocycle $\psi$ determines a normalized one $-\psi$, and vice versa. Therefore, we may reduce, without loss of generality, to the case of  normalized cocycles.


The set of cocycles forms an abelian group $Z(G)$ under pointwise multiplication, and the coboundaries form a subgroup $B(G)$. A basis $ {\bf B}$ for cocycles over $G$ consists of some elementary coboundaries $\partial_i$ and some representative cocycles, so that every cocyclic matrix admits a unique representation as a Hadamard (pointwise)  product $M=M_{\partial_{i_1}}\circ \ldots\circ M_{\partial_{i_w}}\circ R$, in terms of some coboundary matrices $M_{\partial_{i_j}}$ and a matrix $R$ formed from representative cocycles.

Recall that every {\em elementary coboundary} $\partial_d$ is constructed from the characteristic set map  $\delta_d\colon G\rightarrow  \{-1,1\}$ associated with an element $g_d\in G$, so that
$$\partial_d(g_i,g_j)=\delta_d(g_i)\delta_d(g_j)\delta_d(g_ig_j)\quad \mbox{for}\quad \delta_d(g_i)=\left\{\begin{array}{rr} -1 & g_d=g_i,\\
1 & g_d\neq g_i.
\end{array}\right.$$

\begin{lemma}[Lemma 1 \cite{AAFR08}]\label{notacob}

In particular, for $d\neq 1$, every row $s \notin \{ 1,d\}$ in $M_{\partial _d}$ contains precisely two $-1$s, which are located at the positions $(s,d)$ and $(s, e)$, for $g_e = g^{-1}_s g_d$. Furthermore, the first row is always formed by $1$s, while the $d$-th row is formed all by $-1$s, excepting the positions
$(d, 1)$ and $(d, d)$.\end{lemma}

Although the elementary coboundaries generate the set of all coboundaries, they might not be linearly independent (see \cite{AAFR09} for details).

Let $G_r(M)$ (resp. $G_c(M)$) be the Gram matrix of the rows (resp. columns) of $M$,
$$G_r(M)=MM^T,\quad (\mbox{resp.} \,G_c(M)=M^TM).$$
The Gram matrices of a cocyclic matrix can be calculated as follows.
\begin{proposition}[Lemma 6.6 \cite{Hor07}]\label{prop1} 
Let $M_\psi$ be a cocyclic matrix,
\begin{equation}\label{aatc}[G_r(M_\psi)]_{ij}=\psi(g_ig_j^{-1},g_j)\sum_{g\in G}\psi(g_ig_j^{-1},g),
\end{equation}
\begin{equation}
\label{atac}
[G_c(M_\psi)]_{ij}=\psi(g_i,g_i^{-1}g_j)\sum_{g\in G}\psi(g,g_i^{-1}g_j).
\end{equation}
\end{proposition}

If a cocyclic matrix $M_\psi$ is Hadamard,  the cocycle involved, $\psi$, is said to be orthogonal and $M_\psi$ is {\it a cocyclic Hadamard matrix}. The cocyclic Hadamard test asserts that a normalized cocyclic matrix is Hadamard if and only if every  row sum (apart from the first) is zero \cite{Hor07}. In fact, this is a straightforward consequence of Proposition \ref{prop1}.

Analyzing this relation from a new perspective, one could think of normalized cocyclic matrices meeting Hadamard's bound  as normalized cocyclic matrices for which every row sum is zero. Could it be possible that such a relation translates somehow to the case $n\equiv 2 \mod 4$? We  proved in \cite{AAFG12} that the answer to this question is affirmative.

A natural way to measure if the rows of a normalized cocyclic matrix $M=[m_{ij}]$ are close to sum zero, is to define an {\it absolute row excess} function $RE$, such that
$$RE(M)=\sum_{i=2}^n \left| \sum_{j=1}^n m_{ij} \right|.$$
This is a natural extension of the usual notion of {\em excess} of a Hadamard matrix, $E(H)$, which consists in the summation of the entries of $H$.

With this definition at hand, it is evident that a cocyclic matrix $M$ is Hadamard if and only if $RE(M)=0$. That is, a cocyclic matrix $M$ meets Hadamard's bound if and only if $RE(M)$ is minimum. This condition may be generalized to the case $n \equiv 2 \mod 4$.

For the remainder of the paper $t$ denotes an odd positive integer.

\begin{proposition} [\cite{AAFG12}]\label{prophadi}
Let $M$ be a normalized cocyclic matrix over $G$ of order $n=2t$. Then
 $RE(M)\geq 2t-2$.
\end{proposition}

But we may go even further. Having the minimum possible value $2t-2$ is a necessary condition for a cocyclic matrix $M$ to meet the bound $(\ref{ewb1})$.

\begin{proposition}[\cite{AAFG12}] \label{propvic}
If a cocyclic matrix $M$ of order $n=2t$ meets the bound (\ref{ewb1}), then $RE(M)=2t-2$.
\end{proposition}

Unfortunately, although having minimum absolute row excess is a necessary and sufficient condition for meeting  Hadamard's bound , it is just a necessary (but not sufficient, in general, see \cite[Table 5]{AAFG12}
) condition for meeting the bound (\ref{ewb1}). But there are some empirical evidences that matrices having minimum absolute row excess correspond with matrices having large determinants, see \cite[Table 2.1., page 11]{AAFG12b}.

From now on, we fix $G={ D}_{2t}$ as the dihedral group with presentation $\langle a,b\colon \,a^t=b^2=(ab)^2=1\rangle$, with ordering $\{1,a,\ldots,a^{t-1},b,ab,\ldots,a^{t-1}b\}$  and indexed as $\{1,\ldots,2t\}$ for $t$ an odd positive integer.
  A basis for cocycles over ${ D}_{2t}$ consists in (see \cite{AAFG12,AAFR08}):

   $${\bf B}=\{ \partial _2, \ldots , \partial_{2m-1},\beta\}.$$

  ere $\partial_i$ denotes the coboundary associated with the $i^{th}$-element of the dihedral group ${ D}_{2t}$, that is $a^{i-1 \, (\mbox{\tiny mod} \, t)}b^{\lfloor \frac{i-1}{t}\rfloor}$. And $\beta$  is the representative cocycle in cohomology, i.e. the cocyclic matrix coming from inflation is  $M_{\beta}=\left[ \begin{array}{rr} 1 & 1 \\ 1 & - \end{array}\right]\otimes J_t$.
  We use $A\otimes B$ for denoting the usual Kronecker product of matrices, that is, the block matrix whose blocks are $a_{ij}B$.

  Summarizing, the  $D_{2t}$-matrices of the form $M=M_{{\partial}_{i_1}}\circ \ldots\circ M_{{\partial}_{i_w}}\circ M_{{\beta}}$ with $RE(M)=2t-2$ yield a potential  source of matrices with large determinant. In \cite{AAFG12,AAFG12b} we have performed exhaustive searches for the set of $D_{2t}$-matrices with large determinant, for $3\leq t\leq 11$. In the sequel, we will use these matrices  for our purpose of looking for matrices of skew type with large determinants.

\section{Finding equivalent matrices of skew type}


We perform a backtracking search to decide whether or not a $(-1,1)$-matrix  $M$ of size $n$ (given as input) is equivalent to a matrix of skew type. If so,  a skew-matrix $K$ equivalent to $M$ is provided. Starting from the $1\times 1$ matrix, $[1]$, the candidate matrix is built up by ``skew-symmetrically'' appending one row and column at a time, until size $n$ is reached or no continuation is possible. At this stage, the algorithm returns to the most recent sub-matrix from which there is a possible continuation that has not yet been tried, and resumes the search from there. Proceeding exhaustively in this manner, until either the candidate matrix has size $n$ or the entire search tree has been explored, the algorithm terminates stating whether or not $M$ is equivalent to a skew type. If so, providing $K$.

Naturally, the search tree is vast, and various methods must be used to prune it. The value $(n!)^2$ bounds the space we need to explore (in the worst case). Proposition \ref{rowopsk} provides a criterion for removing branches of the search tree reducing it to $n!$. In our observation, most of the branches of the search tree are removed at an early stage. So we can claim that our algorithm works efficiently for the values of $n$ that we have worked on in this manuscript. A deeper study of the time-complexity of this algorithm and efforts to try to improve it are underway.

\begin{proposition}\label{rowopsk}
Let $M$ be a $(-1,1)$-matrix and $K$ be a $(-1,1)$-matrix of skew type. If $M$ and $K$ are equivalent then a $(-1,1,0)$-monomial matrix $Q$ exists such that  $Q^T K Q$ can be obtained from $M$ by a sequence of row operations (interchanges and negations).
\end{proposition}
 \begin{proof}
If $K$ and $M$ are equivalent then there exist $(-1,1,0)$-monomial matrices $P$ and $Q$ such that $PMQ^T=K$. If $K$ is of skew type then
$Q^T K Q$ is so. Let us observe that $ Q^T K Q= Q^TP M$ and $Q^TP$ is a $(-1,1,0)$-monomial matrix.
  \end{proof}

\vspace{2mm}

\noindent{ \bf Algorithm.}  Search for  matrices of skew type equivalent to $M$.

\vspace{2mm}

\noindent{Input: a $(-1,1)$-matrix $M$ of order $n$.}
\newline
\noindent{Output: a   matrix $K$ of skew type equivalent to $M$, if such a matrix $K$   exists. }

\vspace{3mm}
 $r_i$ denotes the $i$-th row of $M$ for $1\leq i\leq n$.
\vspace{1mm}
\begin{itemize}
\item[1.] Initialize variables.
\newline
$\begin{array}{ll}
s=1 & \quad \mbox{order of the current sub-matrix}\\[2mm]
K_1^0=[1] & \quad\mbox{the initial sub-matrix}\\[2mm]
F_1^0=\{r_1\} & \quad\mbox{a list of rows of $M$ or $-M$ which take part in the sub-matrix}\\[2mm]
\bar{F}_1^0=\{r_2,\ldots,r_n\} & \quad\mbox{a list of rows of $M$  which do not belong to $F_1^0$ }
\end{array}$
\item[2.] Iteration $l+1$

$\begin{array}{l}
K_s^l=[r^l_{ij}]_{1\leq i,j\leq s}\\[2mm]
F_s^l=\{r_1^l,r_2^l,\ldots,r_s^l\}\\[2mm]
\bar{F}_s^l=\{r^l_{i_1},\ldots,r^l_{i_k}\}\subset \{r_{s+1}^l,\ldots,r_n^l\}
\end{array}$

\begin{itemize}

\item[2.1.] If $s=n$ then the Output is $K=K_s$ and EndAlgorithm.

\item[2.2.] If $0\leq s<n$ then
   \begin{itemize}
   \item[2.2.1.] If $\bar{F}_{s}^l$ is not empty (i.e., $k>0$) then,
   $$\begin{array}{l}
   r_{i_1}^{l+1}=\left\{\begin{array}{cl}
                          r_{i_1}^l & \mbox{if}\,r_{i_1\;s+1}=1\\
                           -r_{i_1}^l & \mbox{if}\,r_{i_1\;s+1}=-1\end{array}\right.   \end{array}$$

            \begin{itemize}
            \item[2.2.1.1.] If \begin{equation}\label{skewp}
            r^{l+1}_{i_1,j}=-r^l_{j,s+1}\, \forall j=1,\ldots, s;\end{equation}
            then
            $$                                 K_{s+1}^{l+1}=\left[\begin{array}{cc}
                     K_{s}^{l} & \begin{array}{c} r^l_{1,s+1}\\ \vdots \\ r^l_{s,s+1}\end{array}\\
                     r^{l+1}_{i_1,1}\;\ldots\;r^{l+1}_{i_1,s} & 1
            \end{array}\right]$$
            $\begin{array}{l}
            r_j^{l+1}=r^l_j \,\forall 1\leq j\leq n\,\,\mbox{and}\,\,j\neq s+1\\
            r_{s+1}^{l+1}=r_{i_1}^{l+1}\\
                        F_{s+1}^{l+1}=\{r_1^{l+1},\ldots,r_{s+1}^{l+1}\}\\            \bar{F}_{s+1}^{l+1}=\{r_{s+2}^{l+1},\ldots,r_n^{l+1}\}\\
                        \bar{F}_s^{l+1}=\bar{F}_s^l,\bar{F}_{s-1}^{l+1}=\bar{F}_{s-1}^l,\ldots, \bar{F}_1^{l+1}=\bar{F}_1^l\\
                        s=s+1,\,l=l+1 \,\,\mbox{ and go to 2.}
                        \end{array}$
             \item[2.2.1.2.] If identity (\ref{skewp}) does not hold then

             $\begin{array}{l}
             F_s^{l+1}=F_s^{l}\\
             \bar{F}_s^{l+1}=\bar{F}_s^{l}\setminus \{r^l_{i_1}\}\\
             i_h=i_{h+1},\,\,1\leq h\leq k-1\\
             r_{i}^{l+1}=r_{i}^l,\,\, 1\leq i\leq n\\
             K_s^{l+1}=K_s^{l}\\
             l=l+1 \,\,\mbox{and go to 2.2.}
             \end{array}$
            \end{itemize}

     \item[2.2.2.] If $\bar{F}_{s}^l$ is  empty (i.e., $k=0$) then,
       \begin{itemize}
       \item[2.2.2.1] If $s>0$ then

     $\begin{array}{l}
          F_{s-1}^{l+1}=F_s^l\setminus \{r_s^l\}\\
          r_i^{l+1}=r_i^l,\,\,1\leq i\leq n\\
     K_{s-1}^{l+1}=[r_{i,j}^{l+1}]_{1\leq i,j\leq s-1}\\
     \bar{F}_{s-1}^{l+1}=\bar{F}_{s-1}^l\setminus \{r_s^l\}\\
                        s=s-1,\,l=l+1 \,\,\mbox{ and go to 2.}
     \end{array}$
       \item[2.2.2.2.] If $s=0$ then, EndAlgorithm and $M$ is not equivalent to a matrix $K$ of skew type.
       \end{itemize}

   \end{itemize}

\end{itemize}

\end{itemize}

The following result states an independency of the representative  of skew type chosen in the same equivalence class with respect to the function $h(K)=\det(K-I)$. It will play an essential role to study {\it Lin's correspondence} in the next section.

\begin{proposition}
Let $K_1$ and $K_2$ be $(-1,1)$-matrices of skew type. If $K_1$ and $K_2$ are equivalent then $\det (K_1-I)=\det (K_2-I).$
\end{proposition}
\begin{proof}
If $K_1$ and $K_2$ are equivalent then there exist $(-1,1,0)$-monomial matrices $P$ and $Q$ such that  $ Q^T K_1 Q= Q^TP K_2$ (see the proof of Proposition \ref{rowopsk}). Let us point out that $Q^T K_1 Q$ is of skew type.   Due to the fact that $Q^TP K_2$ and $K_2$ are of skew type, it follows  $Q^TP=I$ and hence $K_2=Q^T K_1 Q$. Now, by a simple inspection, we have
$K_2-I=Q^T (K_1-I) Q$ and from here, the desired result.
\end{proof}

 \section{Explicit calculations}
The problem of the spectrum (or range) of the determinant function was studied by Craigen \cite{Cra90} who asked for the complete list of integers, $d$, such that
$d$ is the determinant of some $(0,1)$-matrix of size $n-1$, or equivalently, $2^{n-1} d$ is the determinant of some $(-1,1)$-matrix of size $n$. We are focusing here on $D_{2t}$-matrices.

 We have performed an exhaustive search  to determine the complete range of the determinant function for  $D_{2t}$-matrices, for $3\leq t\leq 11$ odd. For each value of the range $2^{2t-1}d$, we have checked if there is a $(-1,1)$-matrix of skew type $K$ equivalent to $M$, a $D_{2t}$-matrix with determinant equal to $2^{2t-1}d$. If so, we have computed $\det (K-I)$. In Table \ref{tab-01}, $R_M$ and $R_K$ denote the ratios $\displaystyle \frac{\det M}{(2n-2)(n-2)^{\frac{1}{2}n-1}}$ and $\displaystyle \frac{\det (K-I)}{(2n-3)(n-3)^{\frac{1}{2}n-1}}$, respectively.
   For convenience, we will treat  only the values $2^{n-1} d$ in $[2(n+2)^{\frac{1}{2}n-1},(2n-2)(n-2)^{\frac{1}{2}n-1}]$  in this discussion.


\begin{table}
\def~{\phantom0}
\catcode`\@=13
\def@{\phantom.}
\caption{Range for $D_n$-matrices with $n=6, 10, 14, 18$ and $22$ in the interval $[2(n+2)^{\frac{1}{2}n-1},(2n-2)(n-2)^{\frac{1}{2}n-1}]$.\label{tab-01}}
\medskip
\begin{center}
\begin{tabular}{l|ccccc}
\hline
\multicolumn1l{\it  }  & $\det(M)/2^{n-1}$ & $R_M$ & Skew &  $\det(K-I)$ & $R_K$ \\
\hline
\hline
\multicolumn1l {\it n=6}  \\\hline
2&5& 1&\mbox{Yes}&81 & 1\\
1&4&0.8 &\mbox{Yes}&49 &0.605 \\\hline
\multicolumn1l{\it n=10}   \\
\hline
6&144& 1&\mbox{No}& & \\
5&125& 0.868&\mbox{Yes}&33489 & 0.82\\
4&81&0.563 &\mbox{Yes}&14641 & 0.359 \\\hline
\multicolumn1l{\it n=14}   \\
\hline
15 & 9477  & 1 & \mbox{Yes} & 44289025 &  1\\
14 & 8405  & 0.887 & \mbox{Yes} & 38155329 &  0.862\\
13 & 7569  & 0.799 & \mbox{No} &  &  \\
12 & 4096  & 0.432 & \mbox{Yes} & 11390625 &  0.257\\ \hline

\multicolumn1l{\it n=18}   \\
\hline
74&1114112& 1&\mbox{No}& &\\
73&1003520&0.901&\mbox{No}& &\\
72&998001&0.896 &\mbox{No}& &\\
71&950480& 0.853&\mbox{Yes}&70084620225&0.829\\
70&912925&0.819 &\mbox{Yes}&66721473025&0.789\\
69&842724&0.756 &\mbox{No}& &\\
68&812500&0.729 &\mbox{Yes}&57631204225&0.681\\
67&426320&0.383 &\mbox{Yes}&28067976225&0.332\\
66&411892&0.370 &\mbox{Yes}&27048736225&0.320\\65&390625&0.351 &\mbox{Yes}&16983563041&0.201\\
\hline

\multicolumn1l{\it n=22}   \\
\hline
195& 184769649 &0.901&\mbox{No}& &\\
194&179802493  &0.877   &\mbox{Yes}&216409254831025 &0.861\\
193& 173102177  & 0.844&\mbox{No}& &\\
192&164795405& 0.804&\mbox{Yes}& 195146846433009& 0.776\\
191&158835609&0.775 &\mbox{No}& & \\
190&149309173&0.728 &\mbox{Yes}&173517785938225 &0.690\\
189&109098025&0.532 &\mbox{Yes}&125239219607089 & 0.498\\
188&97726205&0.477 &\mbox{Yes}& 110427668520849&0.439 \\
187&95262037&0.465 &\mbox{Yes}&107496564483025&0.428\\
186&90792257&0.443 &\mbox{No}& &\\
185& 90269001 &0.440&\mbox{No}& &\\
184&80120045 &  0.391 &\mbox{Yes}&87838151584401 &0.349\\
183& 74701449 &0.364 &\mbox{No}& &\\
182&72999936& 0.356&\mbox{No}& &\\
181&71233553&0.347&\mbox{No}& &\\
180&70725605&0.345 &\mbox{Yes}&75040101679761 &0.299\\
179&69900605&0.341 &\mbox{Yes}&74857259736081&0.298\\
178&68865749&0.336 &\mbox{No}& &\\
177&68204153&0.333 &\mbox{No}& &\\
176&66810757&0.326 &\mbox{Yes}&69917274339025&0.278 \\
175&62693405&0.306 &\mbox{Yes}&64974062301201&0.258 \\
174&60466176&0.295 &\mbox{Yes}&41426511213649& 0.165\\
\hline

\end{tabular}
\end{center}
\end{table}

Looking at Table \ref{tab-01}, we observe
a regularity in the growth of the functions $\det K$ and
$\det (K-I)$. As a consequence, the largest value of $\det(K-I)$ corresponds to $K$ where the largest value of $\det(K)$ is reached and vice versa. This weighs heavily in favor of Lin's correspondence. For values of the determinant,  $2^{n-1} d$, lesser than $2(n+2)^{\frac{1}{2}n-1}$ no regularity is observed.

\section{Conclusions}
We have approached the problem of the maximal determinant  for $(-1,1)$-matrices of skew type using $D_{2t}$-cocyclic matrices for $t=3,5,7,9$ and $11$. As an intermediate needed step, we have designed a procedure to decide whether or not  a $(-1,1)$-matrix $M$  (given  as an input)  is equivalent to a matrix of skew type. If so, it provides such a matrix of skew type, $K$. Furthermore, we have also computed $\det (K-I)$.

Let $g_k^c(2t)$ denote the maximum determinant of all the $(-1,1)$-matrices of skew type equivalent to a $D_{2t}$-cocyclic matrix. Respectively, let $f_k^c(2t)$ denote   the maximum determinant of $K-I$, all skew matrices with elements $0$ on the diagonal and $\pm 1$ elsewhere such that $K$ is equivalent to a  $D_{2t}$-cocyclic matrix.
Our empirical data suggest that if  $K$ is  a $(-1,1)$-matrix of skew type equivalent to a $D_{2t}$-cocyclic matrix, then
$$\det K=g_k^c(2t)\Longleftrightarrow \det(K-I)=f_k^c(2t).$$
This weighs heavily in favor of Lin's correspondence.
By definition, we have $g_k^c(2t)\leq g_k(2t)$ and $f_k^c(2t)\leq f_k(2t)$.

 Table \ref{tab-03}  shows the greatest values for the determinant of $(-1,1)$-matrices of skew type that we have computed using our cocyclic approach as well as the greatest values for the determinants of skew matrices $K-I$.   For the cases where $g_k(n)$ and $f_k(n)$ is known, $n=6,10,14$ and $26$, we have got $g^c_k(n)=g_k(n)$ and $f^c_k(n)=f_k(n)$, respectively. For $n=18$ and $n=22$, we believe that they are new records.

\begin{table}
\def~{\phantom0}
\catcode`\@=13
\def@{\phantom.}
\caption{Values of $g^c_k(n)$ and $f^c_k(n)$ up to $n=26$.\label{tab-03}}
\medskip
\begin{center}
\begin{tabular}{l|cc}
\hline
\multicolumn1l{\it n} & $g^c_k(n)/2^{n-1}$ &   $f^c_k(n)$ \\
\hline\hline
6                    & 5 & 81  \\
10                    & 125 & 33489  \\
14                   & 9477 & 44289025  \\
18        &  950480 &  70084620225 \\
22            &  179802493 &   216409254831025\\
26         & 54419558400  & 1073816597168995729\\
\hline
\end{tabular}
\end{center}
\end{table}

\section*{Acknowledgements}

The authors would like to thank Kristeen Cheng for her reading of this manuscript.


\end{document}